\documentclass{amsart}
\usepackage{xypic}

\def\AA{\mathbf{A}}
\def\CC{\mathbf{C}}
\def\ZZ{\mathbf{Z}}
\def\GL{\operatorname{GL}}
\def\<{\langle}
\def\>{\rangle}
\def\Hilb{\operatorname{Hilb}}
\def\K{\mathcal{K}}

\newtheorem{theorem}{Theorem}[section]
\newtheorem*{theorem*}{Theorem}
\newtheorem{proposition}[theorem]{Proposition}
\newtheorem{lemma}[theorem]{Lemma}
\newtheorem{corollary}[theorem]{Corollary}
\theoremstyle{definition}

\newtheorem{remark}[theorem]{Remark}

\title{Equivariant Chow classes of matrix orbit closures}
\author{Andrew Berget} 
\address{Department of Mathematics\\
Western Washington Universty\\
Bellingham, WA}
\email{andrew.berget@wwu.edu}

\author{Alex Fink}
\address{School of Mathematical Sciences\\
  Queen Mary University of London\\
  United Kingdom}
\email{a.fink@qmul.ac.uk}

\date{\today}
\begin{document}
\maketitle

\begin{abstract}
  Let $G$ be the product $\GL_r(\CC) \times (\CC^\times)^n$. We show
  that the $G$-equivariant Chow class of a $G$ orbit closure in the
  space of $r$-by-$n$ matrices is determined by a matroid. To do this,
  we split the natural surjective map from the $G$ equvariant Chow
  ring of the space of matrices to the torus equivariant Chow ring of
  the Grassmannian. The splitting takes the class of a Schubert
  variety to the corresponding factorial Schur polynomial, and also
  has the property that the class of a subvariety of the Grassmannian
  is mapped to the class of the closure of those matrices whose row
  span is in the variety.
\end{abstract}

\section{Introduction}
The first goal of this paper is to prove that the Chow class of a
certain affine variety determined by a $r$-by-$n$ matrix is a function
of the matroid of that matrix.  Specifically, given an $r$-by-$n$
matrix $v$ with complex entries, we let $X_v^\circ$ denote the set of
those matrices that are projectively equivalent to $v$ in the sense
that they are of the form $g v t^{-1}$, where $g \in \GL_r(\CC)$, and
$t \in \GL_n(\CC)$ is a diagonal matrix. Let $G$ be the group
consisting of pairs of matrices $(g,t)$, which acts on the space
$\AA^{r \times n}$ of $r$-by-$n$ matrices via the rule $(g,t)v =
gvt^{-1}$. 
A matrix orbit closure $X_v$ is the Zariski closure of $X_v^\circ$ in
$\AA^{r \times n}$; it is the $G$ orbit closure of $v$. This variety
determines a class in the $G$ equivariant Chow ring of $\AA^{r \times
  n}$. Theorem~\ref{thm:matroid inv} states that this class depends
only on the matroid of $v$.

This matroid invariance is a consequence of two results. The first
result is the matroid invariance of the class of a torus orbit closure
in the torus equivariant $K$-theory of the Grassmannian $G(r,n)$. This
result was shown by Speyer \cite{speyer} and was used by Speyer and
the second author to find a purely algebro-geometric interpretation of
the Tutte polynomial \cite{finkspeyer}. The second result which our
matroid invariance relies on deals with the relationship between the
$G$ equivariant Chow ring of $\AA^{r \times n}$ and the torus
equivariant Chow ring of $G(r,n)$, which we now explain.

The geometry of a particular subvariety $Y$ of the Grassmannian
$G(r,n)$ (or more generally, a partial flag variety) is of interest.
To study it, one constructs a certain matrix analog of~$Y$, defined to
be the closure in $\AA^{r \times n}$ of $\pi^{-1}(Y)$ where $\pi$ is
the projection from the space full rank $r$-by-$n$ matrices to
$G(r,n)$. Let $X$ denote this matrix analog, which is a $\GL_r(\CC)$
invariant subvariety of $\AA^{r \times n}$. Sometimes $X$ can
be effectively studied using the techniques of combinatorial
commutative algebra, in the sense that its prime ideal is recognizable
and a Gr\"obner basis can be produced from it. If the original variety
$Y$ had the action of a subgroup of $\GL_n(\CC)$ (acting on the
Grassmannian in the usual way), then $X$ has the action of the product
of $\GL_r(\CC)$ and this group. Such analogs have been constructed for
Schubert varieties, first by Fulton \cite{fulton} and later by
Knutson, Miller \cite{km} and their collaborators (e.g.,
\cite{kms,kmy}). Knutson considered the matrix analog of Richardson
varieties in \cite{knutson}. Recently, Weyser and Yong have
constructed such analogs for symmetric pair orbit closures in flag
varieties \cite{wy}.

The matrix analog of a torus orbit closure in $G(r,n)$ is precisely a
variety of the form $X_v$, where $v$ has rank $r$. In this case, the
matrix analog appears to be, in some sense, more complicated than the
original variety. Set-theoretic equations are known for $X_v$, but
they are not known to generate its prime ideal.  It is natural to ask
if the apparent added complications are visible in various algebraic
invariants of $X_v$. This is how our second main result of the paper
arose. We will prove the following theorem.
\begin{theorem*}\label{thm:splitting}
  The natural surjective map of $\mathbf{Z}[t_1,\dots,t_n]$-modules,
  \[
  A^*_G(\AA^{r \times n}) \to A^*_T(G(r,n)),
  \]
  has a splitting $s$ that satisfies the following properties: For
  every closed, irreducible, $T$-invariant subvariety $Y \subset
  G(r,n)$,
  \begin{enumerate}
  \item[(i)]  $s([Y]_T) = [\overline{\pi^{-1}Y}]_G$,
  \item[(ii)] $s([Y]_T)$ is a $\ZZ[t_1,\dots,t_n]$-linear combination
    of classes of matrix Schubert varieties $[X_\lambda]_G$.
  \end{enumerate}
\end{theorem*}
The structure of our paper is as follows. In Section~\ref{sec:chow} we
provide the required background on equivariant Chow groups. In
Section~\ref{sec:splitting} we recall results of Feh\'er and Rim\'anyi
used to bound polynomial degrees in Chow classes, and use these
results to prove Theorem~\ref{thm:splitting}, which is the theorem
stated above. In Section~\ref{sec:matrix orbit closures} we use
Theorem~\ref{thm:splitting} to prove Theorem~\ref{thm:matroid inv} on
the matroid invariance of the class of $X_v$. Lastly, in
Section~\ref{sec:uniform} we use equivariant localization to compute
the Chow class of a sufficiently generic torus orbit closure in
$G(r,n)$, and use Theorem~\ref{thm:matroid inv} to compute the Chow
class of $X_v$ when $v$ has a uniform rank $r$ matroid.

\section{Equivariant Chow ring}
\subsection{Background on Chow groups and rings}\label{sec:chow}
A variety is a reduced and irreducible scheme over $\CC$ and a
subvariety is a closed subscheme which is a variety. Let $X$ be a
variety over $\CC$. Assume that $X$ has an action of a reductive
linear algebraic group $G$. Our main references for equivariant Chow
groups are \cite{brion,eg}.

Let $V$ be a representation of $G$ containing an open subvariety $U$
which is the total space of a principal $G$-bundle. Such a
representation always exists because $G$ is reductive. The product $X
\times U$ has a free $G$ action, so the quotient space $X \times_G U :=
(X \times U)/G$ is a variety. Assume the codimension of $U$ in $V$ is
larger than some integer~$k$. The $G$-equivariant Chow group of $X$ of degree $k$
is defined as
\[
A_k^G(X) := A_{k+\dim(V)-\dim(G)}( X \times_G U),
\]
and this is independent of the choice of $U$. Here $A_k(-)$ is the
usual Chow group of dimension $k$ cycles on $-$, modulo rational
equivalence.  If $Y \subset X$ is a $G$-invariant subvariety of
codimension $k$, then $Y$ defines a class $[Y]_G :=[Y \times_G U]$ in
$A^k_G(X)$.

For all integers $k$, there is an exact sequence of groups
\begin{equation}
  A_k^G(Y) \to A_k^G(X) \to A_k^G(X-Y) \to 0
\end{equation}
where the former map is pushforward and the latter is pullback.  In
general, any proper $G$-equivariant map $f: Y \to Z$ gives rise to a
pushforward map $A^G_k(Y) \to A^G_k(Z)$ and any flat $G$-equivariant
map $X \to Y$ gives rise to a pullback map $A^G_k(Y) \to A^G_k(X)$.

\begin{proposition}
  Let $Y \subset X$ be a closed, irreducible, $G$-invariant
  subvariety of dimension $d$. Then,
  \begin{enumerate}
  \item[(i)] $[Y]_G$ freely generates $\ker(A_d^G(X) \to A^G_d(X-Y))$,
  \item[(ii)] for all $j > d$, we have $\ker(A_j^G(X) \to A^G_j(X-Y)) = 0$.
  \end{enumerate}
\end{proposition}
\begin{proof}
  This follows because $A_j^G(Y) =0$ for $j > d$ and because
  $A_d^G(Y)$ is freely generated by $[Y]_G$.
\end{proof}

Assume $X$ is smooth. Write $A^k_G(X) = A_{\dim(X)-k}^G(X)$ and
define $$A^*_G(X) := \bigoplus_{k \geq 0} A^k_G(X).$$ Since $X$ is
smooth, this group can be endowed with the intersection product, for
which the element $[X]_G \in A^0_G(X)$ is a multiplicative
identity.  The group $A^*_G(X)$ becomes a graded commutative ring
called the $G$-equivariant Chow ring of $X$. This name reflects the
fact that $A^*_G(X)$ is the Chow ring of $X\times_G U$.

When the open complement $X-Y \subset X$ is smooth, one obtains a
surjective map of graded rings $A_G^*(X) \to A_G^*(X-Y)$.
\begin{corollary}
  Suppose that $Y \subset X$ is an irreducible $G$-invariant  subvariety of
  codimension $k$ with a smooth complement $X-Y$. Then the kernel of
  the pullback $A^*_G(X) \to A^*_G(X-Y)$ is a graded ideal satisfying:
  \begin{enumerate}
  \item[(i)] $[Y]_G$ freely generates $\ker(A^k_G(X) \to A^k_G(X-Y))$, and
  \item[(ii)]for all $j<k$, $\ker(A_G^j(X) \to A_G^j(X-Y)) = 0$.
  \end{enumerate}
\end{corollary}
\begin{remark}
  The ideal $\ker(A^*_G(X) \to A^*_G(X-Y))$ is not necessarily
  principal.
\end{remark}
\subsection{$K$-theory and Chow groups of affine spaces} We will
briefly need the torus equivariant $K$-theory of an affine space $\AA$
and its relation to the equivariant Chow ring.

Let $K_0^G(X)$ denote the Grothendieck group of $G$-equivariant
coherent sheaves on $X$. When $X$ is smooth, this is generated by the
classes of locally free sheaves, and this group becomes a ring with
product being the tensor product of locally free sheaves.

When $X = \AA$ is an affine space, then $K_0^G(\AA)$ is simply
$K_0^G(\rm pt)$ which is the representation ring of the group
$G$. The class of a representation corresponds to a trivial bundle
over $\AA$ with $G$ action determined by the representation. If $G$ is
a torus $(\CC^\times)^m$ then $K_0^G(\AA)$ is a Laurent polynomial ring in $m$
variables $\ZZ[t_1^{\pm 1},\dots,t_m^{\pm 1}]$.  Similarly, the equivariant Chow ring
of $\AA$ is $\ZZ[t_1,\dots,t_m]$. Here, a trivial line bundle twisted
by a character is mapped to its first equivariant Chern class.

If $Y \subset \AA$ is a subvariety of $\AA$ then we write $\K(Y)$ for
the class of the structure sheaf of $Y$ in $K_0^G(\AA)$. There is a
recipe to obtain $[Y]_G$ from $\K(Y)$ {\cite[Proposition~1.9]{kms}}.
\begin{proposition}[Knutson--Miller--Shimozono]\label{prop:kms}
  To obtain $[Y]_G$ from $\K(Y)$, first replace each $t_i$ with
  $1-t_i$ and expand the result as a formal power series in the
  $t_i$. Gather the monomials of lowest possible total degree, which
  will be the codimension of $Y$ in $\AA$. The result is $[Y]_G$.
\end{proposition}

\section{Splitting of the localization sequence}\label{sec:splitting}
We now specialize the set-up of Section~\ref{sec:chow} to our main
case of interest. Let $\AA^{r \times n}$, $r \leq n$, be the affine
space of $r$-by-$n$ matrices with entries in $\CC$. This has an action
by $G = \GL_r(\CC) \times T$, where $T = (\CC^\times)^n$ is the
algebraic $n$-torus acting by $(g,t) \cdot m = g m t^{-1}$, viewing $t
\in T$ as a diagonal matrix. For the remainder of our work $G$ will
denote this product of groups.

\subsection{Degree bound of Feh\'er and Rim\'anyi}
The equivariant Chow ring of $\AA^{r \times n}$ is the equivariant
Chow ring of a point, since $\AA^{r \times n}$ is a vector bundle over
a point. We can succinctly describe this object
\cite[Proposition~6]{eg}: It is the ring of Weyl group invariants of
the polynomial ring over the character lattice of a maximal torus of
$G$. Specifically,
\[
A^*_G(\AA^{r\times n}) =
\mathbf{Z}[ u_1,\dots,u_r,t_1,\dots,t_n]^{S_r},
\]
where the symmetric group $S_r$ acts by permuting the subscripts on
the $u$~variables. Here the $u$~variables represent the characters of
the diagonal torus in $\GL_r(\CC)$ and the $t$~variables represent the
characters of the torus $T$.

We consider the map $\AA^{r \times n} \to \AA^{(r-1) \times n}$ that
forgets the last row of a matrix. Given $Y \subset \AA^{r \times n}$
we let $Y_0$ denote the image of $Y$ under the above map, which is
itself closed and irreducible.

The following result relates the classes of $Y$ and $Y_0$ and is a
special case of a more general result due to Feh\'er and Rim\'anyi
{\cite[Theorem~2.1]{fr07}}.
\begin{theorem}\label{thm:fr07}
  Let $Y$ and $Y_0$ be as above, which we take to have codimension $c$
  and $c_0$, respectively. Write $[Y]_G \in A^{c}_G(\AA^{(r+1) \times
    n})$ uniquely as
  \[
  [Y]_G = \sum_{k} p_k(u_1,\dots,u_{r-1},t_1,\dots,t_n) \cdot
  u_r^{c-k},
  \]
  where $p_k$ is a homogeneous polynomial of degree $k$. Write $G_0 =
  \GL_{r-1}(\CC) \times T \subset G$.  Then for all $k \geq 0$, $p_k$
  is in the kernel of $A^*_{G_0}(\AA^{(r-1)\times n}) \to
  A^*_{G_0}(\AA^{(r-1)\times n} - Y_0)$. In particular, the degree of
  $u_r$ in $[Y]_G$ is at most $c-c_0$.
\end{theorem}
\begin{corollary}\label{cor:width bound}
  Using the notation above, if $Y$ is the closure of its full rank
  matrices then the degree of $u_r$ in $[Y]_G$ is at most $n-r$.
\end{corollary}
\begin{proof}
  It suffices to verify that $c-c_0 \leq n-r$, and this is equivalent
  to $\dim(Y) - \dim(Y_0) \geq r$. To verify this, we consider the
  non-empty open subset of $Y_0$ of rank $r-1$ matrices. The fiber of
  the natural map $Y \to Y_0$ over a matrix of rank $r-1$ is at least
  $(r-1)+1$ dimensional, since $Y$ is $\GL_r(\CC)$ invariant. Since
  the dimension of this general fiber is precisely $\dim(Y) -
  \dim(Y_0)$, we are done.
\end{proof}

\subsection{Matrix Schubert varieties} The \textbf{Schubert varieties}
of the Grassmannian $G(r,n)$ are $B \subset \GL_n(\CC)$ orbit
closures, where $B$ is a Borel subgroup. Fixing such a $B$, the
Schubert varieties are in bijection with partitions $\lambda =
(\lambda_1 \geq \dots \geq \lambda_r \geq 0)$ with $\lambda_1 \leq
n-r$. We denote the Schubert variety corresponding to $\lambda$ by
$\Omega_\lambda$.  Since the Schubert varieties $\Omega_\lambda$ form
a stratification of $G(r,n)$, the classes $[\Omega_\lambda]_T$ form a
$\mathbf{Z}[t_1,\dots,t_n]$-linear basis of $A^*_T(G(r,n))$.

For $r< n$, denote the set of \textit{f}ull \textit{r}ank $r$-by-$n$
matrices by $\AA^{\rm fr}$. Let $\pi : \AA^{\rm fr} \to G(r,n)$ denote
the projection map, which sends a matrix to the span of its
rows. Define a \textbf{matrix Schubert variety} as $X_\lambda = \overline{\pi^{-1}(\Omega_\lambda)}$, where the closure takes place within
$\AA^{r \times n}$. The equivariant Chow classes of these varieties
were computed by Knutson, Miller and Yong \cite{kmy}. It is important
here that one computes the class of the matrix Schubert variety
$X_\lambda$, and not a \textit{representative} for the class of the
Schubert variety $\Omega_\lambda$. This is to say, one does
computations in $A^*_G(\AA^{r \times n})$ instead of a quotient of
this ring.

Let $\lambda$ be a partition, which we regard simultaneously as a
decreasing sequence of non-negative integers, as above, and as a set
$\lambda = \{c_{ij} : 1 \leq i \leq \lambda_j\}$. We say that $c_{ij} \in
\lambda$ is above (or left) of $c_{k\ell} \in \lambda $ if $j < \ell$ (or $i <
k$). A \textbf{tableau} is a function $\tau: \lambda \to
\mathbf{N}^+$. 

A tableau $\tau: \lambda \to \mathbf{N}^+$ is said to be
\textbf{semistandard} provided that for all $c,d \in \lambda$
\begin{enumerate}
\item[(i)] if  $c$ lays to the left of  $d$ then $\tau(c) \leq \tau(d)$; and
\item[(ii)] if $c$ lays above $d$ then $\tau(c) < \tau(d)$.
\end{enumerate}
We let $SST(\lambda,r)$ be the set of all semistandard tableaux $\tau
: \lambda \to \{1,2,\dots,r\}$. The following result appears
in~\cite[Theorem~5.8]{kmy}, although its origins are much older.
\begin{theorem}[Knutson--Miller--Yong] For any matrix Schubert variety $X_\lambda
  \subset \AA^{r \times n}$,
  \[
  [X_\lambda]_G = \sum_{SST(\lambda,r)} \prod_{c_{ij} \in \lambda} (u_{\tau(c_{ij})} - t_{\tau(c_{ij}) + j-i}).
  \]
\end{theorem}
The displayed polynomial is called a \textbf{factorial Schur
  polynomial}. It is important to note that since the partitions
$\lambda$ above have $\lambda_1 \leq n-r$, the degree of any $u_i$ in
$[X_\lambda]_G$ is at most $n-r$.
\begin{corollary}\label{cor:basis}
  The following is a $\ZZ[t_1,\dots,t_n]$-linear basis for
  $A_G^*(\AA^{r \times n})$: The set of classes of matrix Schubert
  varieties together with the Schur polynomials
  $s_\lambda(u_1,\dots,u_r)$ where $\lambda_1 \geq n-r+1$.
\end{corollary}

\subsection{Splitting of the localization sequence}
In this section we prove the following result.
\begin{theorem}\label{thm:splitting}
  The natural map of $\mathbf{Z}[t_1,\dots,t_n]$-modules,
  \[
  A^*_G(\AA^{r \times n}) \to A^*_G(\AA^{\rm fr}) \approx A^*_T(G(r,n)),
  \]
  has a splitting $s$ that satisfies the following properties: For
  every $T$ invariant subvariety $Y \subset G(r,n)$,
  \begin{enumerate}
  \item[(i)]  $s([Y]_T) = [\overline{\pi^{-1}Y}]_G$,
  \item[(ii)] $s([Y]_T)$ is a $\ZZ[t_1,\dots,t_n]$-linear combination
    of classes of matrix Schubert varieties
    $[X_\lambda]_G$. Equivalently, the Schur polynomial expansion of
    $s([Y]_T)$ is a linear combination of Schur polynomiala
    $s_\lambda(u)$ with $\lambda_1 \leq n-r$.
  \end{enumerate}
\end{theorem}
\begin{proof}
  A splitting $s$ is uniquely determined by condition (ii). So it
  suffices to show that if $[Y]_T = \sum_\lambda q_\lambda
  [\Omega_\lambda]_T$ then $[\overline{\pi^{-1}Y}]_G = \sum_\lambda
  q_\lambda [X_\lambda]_G$.

  Suppose that $Y$ is a $T$-invariant subvariety
  of $G(r,n)$. It follows that $X = \overline{\pi^{-1} Y}$ satisfies
  the hypothesis of Corollary~\ref{cor:width bound} and so the degree
  of $u_r$ (and hence any $u$ variable) in $[X]_G$ is at most
  $n-r$. Hence $[X]_G$ is a $\ZZ[t_1,\dots,t_n]$-linear combination of
  classes of Schur polynomials $s_\lambda$ with $\lambda_1 \leq n-r$
  (\textit{cf.}~ \cite[Theorem~7.4]{fr12}). We conclude that $[X]_G$
  is a linear combination of the the classes of the matrix Schubert
  varieties: $[X]_G = \sum_\lambda p_\lambda [X_\lambda]_G$.

  We can uniquely write
  \[
  [Y]_T = \sum q_\lambda [\Omega_\lambda]_T \in A^*_T(G(r,n)),
  \]
  for some polynomials $q_\lambda \in \ZZ[t_1,\dots,t_n]$. Since
  $\pi^* :A^*_T(G(r,n)) \to A^*_G(\AA^{\rm fr})$ is an isomorphism,
  this yields
  \[
  [\pi^{-1}Y]_G = \sum_\lambda q_\lambda [ \pi^{-1} \Omega_\lambda ]_G.
  \]
  We claim that $[X]_G = \sum_\lambda q_\lambda [ X_\lambda ]_G$.
  To see this, note that
  \[
  \sum_\lambda (q_\lambda-p_\lambda)[X_\lambda]_G \in
  \ker(A_G^*(\AA^{r \times n}) \stackrel{i^*}{\to} A_G^*(\AA^{\rm fr})).
  \]
  Applying $i^*$ to this class gives $\sum_\lambda
  (q_\lambda-p_\lambda)[\pi^{-1} \Omega_\lambda] =0$. However, the
  classes $[\pi^{-1}\Omega_\lambda]_G$ form a
  $\ZZ[t_1,\dots,t_n]$-linear basis for $A^*_G(\AA^{\rm fr})$ so this
  means $q_\lambda = p_\lambda$.
\end{proof}

\section{Matrix orbit closures and matroids}\label{sec:matrix orbit closures}
In this section we prove that the equivariant Chow class of a
$G$-orbit closure in $\AA^{r \times n}$ is determined by a matroid. We
begin by stating the background we need from matroid theory.
\subsection{Matroid terminology}
Write $[n]$ for $\{1,2,\dots,n\}$ and $\binom{[n]}{r}$ for the set of
size $r$ subsets of $[n]$. Let $v \in \AA^{r \times n}$ be any
$r$-by-$n$ matrix. The \textbf{matroid} of $v$, denoted $M(v)$, is the
set of subsets $I \subset [n]$ where the column restricted matrix
$v_I$ has rank $|I|$.  When the rank of $v$ is $k$, so that $k$ is the
maximum cardinality of a set in~$M(v)$, we say that $M(v)$
has rank $k$. In this case $M(v)$ is determined by its size $k$ sets,
which are called its bases.

Given a matroid $M$ with ground set $[n]$ we define its
\textbf{matroid base polytope} $P(M)$ as follows: Let $\{e_i\} \subset
\mathbf{R}^n$ be the standard basis vectors and write $e_I = \sum_{i
  \in I} e_i$. We define $P(M)$ to be the convex hull in
$\mathbf{R}^n$ of $e_B$ where $B$ ranges over the bases of $M$. The
points $e_B$ are actually the vertices of $P(M)$, and the convex hull
of two vertices $e_B$ and $e_{B'}$ forms an edge of $P(M)$ if and only
$B$ and $B'$ differ by exactly one element \cite{Edmonds,GGMS}.

\subsection{Matroid invariance of Chow classes} A \textbf{matrix orbit
  closure} is a $G$ orbit closure of a point in $\AA^{r \times n}$. We
write $X_v$ for the orbit closure of a matrix $v \in \AA^{r \times n}$
and $X_v^\circ$ for the $G$ orbit itself. When $v$ is a rank $r$
matrix we can project $X_v^\circ$ to the Grassmannian. The result
is the $T$ orbit closure of $\pi(v)$, which we denote by $Y_{\pi(v)}$.

The following result is proven by Speyer
\cite[Proposition~12.5]{speyer}.
\begin{theorem}
  For any rank $r$ matrix $v\in \AA^{r \times n}$, the class of the
  structure sheaf of $Y_{\pi(v)}$ in the $T$-equivariant $K$-theory of
  $G(r,n)$ is determined by the matroid $M(v)$.
\end{theorem}
As an immediate corollary we have:
\begin{corollary}
  The $T$-equivariant Chow class $[Y_{\pi(v)}]_T$ is determined by the
  matroid $M(v)$.
\end{corollary}
Applying Theorem~\ref{thm:splitting} gives:
\begin{theorem}\label{thm:matroid inv}
  Assume the matrix $v$ has rank $r$. The class $[X_v]_G$ depends only
  on the matroid  $M(v)$.
\end{theorem}
The case when $v$ has rank less than $r$ can be obtained by taking a
$\operatorname{rank}(v)$-by-$n$ matrix $u$ with the same row span as
$v$, considering $X_u$ in a smaller matrix space and taking the $G$
orbit closure of this variety in $\AA^{r \times n}$ (\textit{cf.}
\cite[Theorem~7.5]{fr12}).

\section{Chow classes for uniform matroids}\label{sec:uniform}
In this section our goal is to explicitly compute $[X_v]_G$ when $v$
is a sufficiently general matrix in $\AA^{r \times n}$. Here,
sufficiently general means that the matroid $M(v)$ is
\textit{uniform}, i.e., no maximal minor of $v$ vanishes. That such a
formula exists is due to Theorem~\ref{thm:matroid inv}. To find this
class, we will follow the idea of Theorem~\ref{thm:splitting}, and
compute the Chow class of the toric variety $Y_{\pi(v)}$ in
$A^*_T(G(r,n))$ and lift the result to $A^*_G(\AA^{r\times n})$.

\subsection{Equivariant localization and the Grassmannian}
In order to state our main result we will need to gather some
background material about equivariant localization and the
Grassmannian.

The Grassmannian $G(r,n)$ has a finite set of $T$-fixed points: they
are the $r$-dimensional coordinate subspaces of $\CC^n$.  We denote by
$x_B$ the fixed point in which the unique coordinate \emph{not} fixed to
zero is the one indexed by the set $B\in\binom{[n]}{r}$.  The Pl\"ucker
embedding embeds $G(r,n)$ equivariantly in $\mathbf{P}^{\binom nr-1} =
\mathbf{P} \bigwedge^r \CC^n$, with $T$ action inherited from the
natural one on $\CC^n$.  A fixed point $x_B$ of~$G(r,n)$ is sent to a
coordinate point in $\mathbf{P}^{\binom nr-1}$, and the character by
which $T$ acts on the corresponding coordinate is $t^B = \prod_{i\in
  B}t_i$.

The inclusion $\iota$ of this discrete fixed set $G(r,n)^T$ into
$G(r,n)$ induces a restriction map
\[
\iota^*:A^*_T(G(r,n)) \to A^*_T(G(r,n)^T).
\]
Its target $A^*_T(G(r,n)^T)$ is a direct sum of polynomial rings 
$A^*_T(\mathrm{pt})=\ZZ[t_1\dots,t_n]$, one for
each fixed point. 
The restriction of the class of a $T$-equivariant subvariety
to a fixed point $x$ will equal the restriction of this class
to an affine space $\AA$ containing $x$ on which $T$ acts linearly,
under the natural isomorphism 
$A^*_T(\AA) = \ZZ[t_1\dots,t_n] = A^*_T({\rm pt})$.
We will let $c|_x\in\ZZ[t_1,\ldots,t_n]$ denote the 
restriction of the class $c$ to~$x$.

Since the Grassmannian is smooth and projective and $T$ is a torus,
results of Brion \cite[Theorems 3.2--3.4]{brion} imply that $\iota^*$
is injective and we can identify the image of $\iota^*$. It consists
of the tuples of polynomials $f = (f_B: B\in\binom{[n]}{r})$ such that
\[
f_B - f_{B \cup j \setminus i} \in \< t_j - t_i \>
\]
for all $i \in B$ and $j \notin B$. Such results were also proved by
Edidin and Graham \cite{eg2} and are closely related to the
topological results of Goresky, Kottwitz and MacPherson \cite{GKM}. 

\subsection{Vector bundles on the Grassmannian}
We will let $\mathcal{S}$ denote the tautological rank $r$ vector
vector bundle over $G(r,n)$. It is a subbundle of the trivial bundle
$\CC^n$ and its fiber over $x \in G(r,n)$ is the $r$-dimensional
subspace $x \subset \CC^n$. The quotient bundle $\mathcal{Q}$ is
$\CC^n/\mathcal{S}$.

When we write a symmetric function of a vector bundle $\mathcal E$, we
mean that symmetric function of its Chern roots, so that $e_k(\mathcal
E) = c_k(\mathcal E)$.  For later reference, we give explicit
expansions of $s_\nu(\mathcal{S}^\vee)$ and $s_\nu(\mathcal Q)$, which
are elements of $A_T^*(G(r,n))$, as polynomials in the variables $u_i$
and $t_j$. The formulae are these:
\begin{align*}
s_\nu(\mathcal \mathcal{S}^\vee) &= s_\nu(  u), \\
s_\nu(\mathcal Q) &= \omega \left( s_\nu(u,t) \right) = \sum_{\lambda,\mu} c_{\lambda,\mu}^\nu s_\lambda(t) s_{\mu'}(u).
\end{align*}
Here $\omega$ is the usual operation on symmetric polynomials that
transposes Schur polynomials, extended $\ZZ[t]$-linearly
\cite[I.2.7]{macdonald}.  On symmetric functions in infinitely many
variables, $\omega$ is an involution; in our setting, it is an
involution as long as no part of a partition exceeds $r$.

If $\mathcal{E}$ is a vector bundle on $G(r,n)$, a Schur polynomial
$s_\lambda$ of the Chern roots of $\mathcal{E}$ localizes at a fixed
point to the sum of the characters by which $T$ acts on the tangent
space of $\mathbf S^\lambda(\mathcal{E}^\vee)$, where $\mathbf
S^\lambda$ is a Schur functor.  For the vector bundles
$\mathcal{S}^\vee$ and $\mathcal{Q}$, the resulting localizations are
\begin{align*}
s_\lambda(\mathcal{S}^\vee)|_{x_B} &= s_\lambda(-t_i : i\in B), \\
s_\lambda(\mathcal{Q})|_{x_B} &= s_\lambda(t_j : j\not\in B).
\end{align*}

\subsection{Statement of the formula}
We will use one piece of notation to succincly state our theorem. The
partition $(n-r-1)^{r-1}$ is the $(r-1)\times(n-r-1)$ rectangle, and
if $\lambda$ and $\mu$ are two partitions, then the Littlewood--Richardson
coefficient $c_{\lambda\mu}^{(n-r-1)^{r-1}}$ equals $1$ or $0$,
according to whether or not $\lambda$ is the $180^\circ$ rotated complement
of $\mu$ within this rectangle. Given a partition $\lambda$
fitting inside a $(r-1) \times (n-r-1)$ box, we let $\tilde\lambda$
denote the unique partition $\mu$ satisfying
$c^{(n-r-1)^{r-1}}_{\lambda\mu} =1$.

\begin{theorem}\label{thm:equivariant uniform class}
  Given $v \in \AA^{r \times n}$ whose matroid is uniform of rank $r$,
  the class of $Y_{\pi(v)}$ in $A^*_T(G(r,n))$ is
  \begin{equation}\label{eq:euc0}
    [Y_{\pi(v)}]_T = \sum_{\lambda \subset (n-r-1)^{r-1}}
    s_\lambda(\mathcal{S}^\vee) s_{\tilde\lambda}(\mathcal Q).
  \end{equation}
  The class of $X_v$ in $A^*_G(\AA^{r\times n})$ is
  \begin{align*}
    [X_v]_G &= \sum_{\substack{\lambda \subset (n-r-1)^{r-1}\\\mu,\nu}} c^{\tilde\lambda}_{\mu\nu} s_\lambda(u) s_{\mu'}(t) s_\nu(u),
    \\&= \omega(s_{(r-1)^{n-r-1}}(u,u,t)).
  \end{align*}
\end{theorem}

\subsection{Proof of the formula}
The first step is to understand the class of $Y_{\pi(v)}$ localized
at a $T$-fixed point.

\begin{lemma}\label{lem:cohom class of T-orbit}
  The $T$-equivariant cohomology class of $Y_{\pi(v)}$ localized at
  $x_B$ is
  \begin{equation}\label{eq:uec1}
    [Y_{\pi(v)}]_T|_{x_B} = \prod_{i \in B, j \notin B} (t_j -t_i) \sum_{(i_1,\ldots,i_n)}
    \frac{1}{(t_{i_2}-t_{i_1})(t_{i_3}-t_{i_2})\cdots
      (t_{i_n}-t_{i_{n-1}})}, 
  \end{equation}
  where the sums range over permutations $(i_1,\ldots,i_n) \in S_n$
  whose lex-first basis is~$B$.
\end{lemma}
Note that the sum occurring in Lemma~\ref{lem:cohom class of T-orbit} 
is zero if $B$ is not a basis of $M(v)$.  When $M(v)$ is uniform, 
the sum ranges over those permutations that have the elements of
$B$ in their first $r$ positions.

\begin{proof} 
  Following the approach of \cite{finkspeyer}, we first identify
  $Y_{\pi(v)}$ as a toric variety.  Viewing toric varieties as images
  of monomial maps \cite[Chapters~7, 10]{miller-sturmfels}, the
  normalization of $Y_{\pi(v)}$ is the toric variety of the polytope
  given as the convex hull of the characters corresponding to the
  $T$-fixed points it contains.  By a result of White
  \cite[Theorem~2]{white-normality}, the variety $Y_{\pi(v)}$ is
  already normal, and therefore \emph{is} the toric variety just
  stated.  The $T$-fixed points in $Y_{\pi(v)}$ are those $x_B$ such
  that $B$ is a basis of the matroid $M(v)$.  The corresponding
  characters are $\{t^B : \textup{$B$ is a basis of }M(v)\}$, whose
  convex hull is the matroid base polytope $P(M(v))$ of $M(v)$,
  defined in Section~\ref{sec:matrix orbit closures}.

  If the toric variety $Y_{\pi(v)}$ contains the fixed point $x_B$,
  then its restriction to the $T$-invariant translate of the big
  Schubert cell around $x_B$ is the corresponding affine patch of
  $Y_{\pi(v)}$, in Fulton's construction: that is, it is the affine
  toric subvariety consisting of the orbits whose closures contain
  $x_B$.  Explicitly, this affine scheme is
  $\operatorname{Spec}\CC[C]$, where $C$ is the semigroup of lattice
  points in the tangent cone to $P(M(v))$ at the vertex $e_B$.  The
  $T$-equivariant $K$-theory class of $[Y_{\pi(v)}]_T|_{x_B}$ is then
  the product of $\Hilb(\CC[C])$ with $\prod_{i\in B,j\not\in B}
  (1-t_j/t_i)$. The Hilbert series $\Hilb(\CC[C])$ is the
  finely-graded lattice point enumerator of $C$.  We claim
  \begin{align}
    \Hilb(\CC[C]) &= \sum_{(i_1,\ldots,i_n)} \Hilb
    \operatorname{cone}(e_{i_2}-e_{i_1},
    \ldots, e_{i_n}-e_{i_{n-1}})\notag\\
    &=\sum_{(i_1,\ldots,i_n)}
    \frac{1}{(1-t_{i_2}/t_{i_1})(1-t_{i_3}/t_{i_2})\cdots
      (1-t_{i_n}/t_{i_{n-1}})}\label{eq:7.2a}
   \end{align}
   where the sums range over permutations $(i_1,\ldots,i_n) \in S_n$
   whose lex-first basis is $B$. To see this, apply Brion's theorem to
   the triangulation of the dual of this cone into type A Weyl
   chambers. The cones in the triangulation are unimodular, and their
   lattice point generators are those given in the second line.

   Altogether,
   \[
   \K(Y_{\pi(v)})|_{x_B}
     = \prod_{i \in B, j \notin B} (1 - t_j/t_i) 
       \sum_{(i_1,\ldots,i_n)}
    \frac{1}{(1-t_{i_2}/t_{i_1})\cdots(1-t_{i_n}/t_{i_{n-1}})}
   \]
   Using Proposition~\ref{prop:kms}, this becomes the equation to be proved 
   upon replacing each $t_i$ with $1-t_i$, and then extracting
   the lowest degree term of the resulting power series.
   (Note that taking the lowest-degree term can be done one factor
   at a time.)
\end{proof}

\begin{proof}[Proof of Theorem~\ref{thm:equivariant uniform class}]
  The second equation of the theorem follows from the first,
  by Theorem~\ref{thm:splitting}.

  By equivariant localization, it is enough to show the claimed
  equality after restriction to each $x_B$, in $A_T^*(x_B) \cong
  \ZZ[t_1,\ldots,t_n]$. On one hand, the restriction of the right side
  of \eqref{eq:euc0} at $x_B$ is
  \[
  \sum_{\lambda} 
  s_\lambda(-t_i : i\in B)\, s_{\tilde \lambda}(t_j : j\not\in B).
  \]
  We massage the formula for $[Y_{\pi(v)}]_T|_{x_B}$ in
  Lemma~\ref{lem:cohom class of T-orbit}, and show that it equals the
  above polynomial.
  
  Let us temporarily write $f(i_1,\ldots,i_n)$ for
  $1/(t_{i_2}-t_{i_1})\cdots(t_{i_n}-t_{i_{n-1}})$.  We have
  \[
  \frac{f(i_1,\ldots,\widehat{i_s},\ldots,i_n)}{f(i_1,\ldots,i_n)} =
  \frac{t_{i_{s+1}}-t_{i_{s-1}}}{(t_{i_s}-t_{i_{s-1}})(t_{i_{s+1}}-t_{i_s})}
  = \frac1{t_{i_{s+1}}-t_{i_s}} - \frac1{t_{i_{s-1}}-t_{i_s}}
  \]
  and similar identities when $s=1$ or $s=n$ in which the right hand term 
  with an out-of-range index in it is deleted.  
  Thus, if $\ell$ is a list of indices, we have a telescoping sum
  \[
  \sum_{\substack{\text{$\ell'$ : $\ell$ is $\ell'$ with $i$ dropped}
      \\ \text{$i$ precedes $j$ in $\ell'$}}} f(\ell') =
  \frac{f(\ell)}{t_j-t_i}.
  \]
  Grouping the terms of the sum in \eqref{eq:uec1} by $i_r$ and
  repeatedly applying the above identity and its order-reversed
  counterpart, we get
  \begin{equation}\label{eq:uec2}
    [Y_{\pi(v)}]_T|_{x_B} = 
    \prod_{i\in B,j\not\in B}(t_j-t_i) \;\cdot\;
    \sum_{i_r\in B}
    \left(\prod_{i\in B\setminus i_r}\frac1{t_{i-r}-t_i}\right)
    \left(\prod_{j\notin B}\frac1{t_j-t_{i_r}}\right).
  \end{equation}

  We next invoke the following variant of the Cauchy identity:
  \begin{align*}
    \prod_{t\in T,v\in V} (t-v)
    &= \sum_{\nu,\mu} c_{\nu\mu}^{(|V|)^{|T|}}
    s_\nu(t\in T)\, s_{\mu'}(-v\in V)
  \end{align*}
  In our localized cohomology class, we combine the first
  and last products in~\eqref{eq:uec2} and apply the Cauchy identity with
  $(T,V)=(\{-t_i:i\in B\setminus i_r\},\{-t_j:j\notin B\})$, giving
  \[
  \sum_{i_r\in B} \left(\prod_{i\in B\setminus
      i_r}\frac1{t_{i_r}-t_i}\right) \!\!
  \left(\sum_{\nu,\mu} c_{\nu\mu}^{(n-r)^{r-1}}
    \frac{\det(-t_i^j)_{i\in B\setminus i_r}^{j=\nu_k+r-1-k}}
    {\det(-t_i^j)_{i\in B\setminus i_r}^{0\leq j<r-1}}
    \cdot s_{\mu'}(t_j : j\not\in B)\right)
  \]
  where the $s_\nu$ is written as a ratio of determinants.  Now we
  combine the remaining product in the above display into the
  Vandermonde determinant in the denominator. The sum over $i_r\in B$
  can then be read as an expansion along the last row of the
  determinantal formula for $s_\lambda(-t_i : i\in B)$, where $\lambda$
  is obtained from $\nu$ by decrementing every part if $\nu$ has $r-1$ parts;
  if $\nu$ has fewer parts then the terms in this determinantal expansion cancel.  
  For a given $\nu$ the only $\mu$ yielding a nonzero term
  is the one such that $c_{\nu,\mu}^{(n-r)^{r-1}}$ equals 1,
  i.e.\ so that $\nu$ and $\mu$ are complements in a $(r-1)\times(n-r)$ rectangle.
  In this event
  $\lambda$ and $\mu$ are complements in a $(r-1)\times(n-r-1)$
  rectangle, so our localized class is
  \[
  \sum_{\lambda \subset (n-r-1)^{r-1}}
  s_\lambda(-t_i : i\in B) s_{\tilde\lambda}(t_j : j\not\in B).
  \]
  This agrees with the localization of~\eqref{eq:euc0} and the theorem
  follows.
\end{proof}

\subsection{Comparison to a formula of Klyachko}
There is another formula, due to Klyachko \cite{klyachko}, for the
non-equivariant class of $Y_{\pi(v)}$ in $A^*(G(r,n))$ when $v$ has a
uniform rank $r$ matroid. 
\begin{theorem}[Klyachko]
  Let $v$ have a uniform, rank $r$ matroid. Let $\lambda \subset
  (n-r)^r$ be a partition of $n-1$. Then the coefficient of
  $[\Omega_\lambda]$ in $[Y_{\pi(v)}]$ is
  \[
  \sum_{i=1}^r (-1)^i \binom{n}{i} s_\lambda(1^{r-i}).
  \]
\end{theorem}
Setting all the $t$ variables equal to zero in
Theorem~\ref{thm:equivariant uniform class} we obtain a different
looking formula for the $\GL_r(\CC)$-equivariant Chow class of $X_v$:
\[
[X_v]_{\GL_r(\CC)} = 
\sum_{\lambda \subset (n-r-1)^{(r-1)}} s_\lambda(u) s_{\tilde\lambda}(u).
\]
As a consequence of this:
\begin{corollary}
  Let $v \in \AA^{r \times n}$ have a uniform rank $r$ matroid. The
  degree of the variety $X_v$ is
  \[
  \sum_{\lambda \subset (n-r-1)^{r-1}} 
    s_\lambda(1^r) s_{\tilde\lambda}(1^r ).
  \]
\end{corollary}

\bibliography{cc}{} \bibliographystyle{alpha}

\def\cprime{$'$} \def\cprime{$'$}
\begin{thebibliography}{GGMS87}

\bibitem[Bri97]{brion}
Michel Brion.
\newblock Equivariant {C}how groups for torus actions.
\newblock {\em Transform. Groups}, 2(3):225--267, 1997.

\bibitem[Edm70]{Edmonds}
Jack Edmonds.
\newblock Submodular functions, matroids, and certain polyhedra.
\newblock In {\em Combinatorial {S}tructures and their {A}pplications ({P}roc.
  {C}algary {I}nternat. {C}onf., {C}algary, {A}lta., 1969)}, pages 69--87.
  Gordon and Breach, New York, 1970.

\bibitem[EG98a]{eg}
Dan Edidin and William Graham.
\newblock Equivariant intersection theory.
\newblock {\em Invent. Math.}, 131(3):595--634, 1998.

\bibitem[EG98b]{eg2}
Dan Edidin and William Graham.
\newblock Localization in equivariant intersection theory and the {B}ott
  residue formula.
\newblock {\em Amer. J. Math.}, 120(3):619--636, 1998.

\bibitem[FNR12]{fr12}
L{\'a}szl{\'o}~M. Feh{\'e}r, Andr{\'a}s N{\'e}methi, and Rich{\'a}rd
  Rim{\'a}nyi.
\newblock Equivariant classes of matrix matroid varieties.
\newblock {\em Comment. Math. Helv.}, 87(4):861--889, 2012.

\bibitem[FR07]{fr07}
L{\'a}szl{\'o}~M. Feh{\'e}r and Rich{\'a}rd Rim{\'a}nyi.
\newblock On the structure of {T}hom polynomials of singularities.
\newblock {\em Bull. Lond. Math. Soc.}, 39(4):541--549, 2007.

\bibitem[FS12]{finkspeyer}
Alex Fink and David~E. Speyer.
\newblock {$K$}-classes for matroids and equivariant localization.
\newblock {\em Duke Math. J.}, 161(14):2699--2723, 2012.

\bibitem[Ful92]{fulton}
William Fulton.
\newblock Flags, {S}chubert polynomials, degeneracy loci, and determinantal
  formulas.
\newblock {\em Duke Math. J.}, 65(3):381--420, 1992.

\bibitem[GGMS87]{GGMS}
I.~M. Gel{\cprime}fand, R.~M. Goresky, R.~D. MacPherson, and V.~V. Serganova.
\newblock Combinatorial geometries, convex polyhedra, and {S}chubert cells.
\newblock {\em Adv. in Math.}, 63(3):301--316, 1987.

\bibitem[GKM98]{GKM}
Mark Goresky, Robert Kottwitz, and Robert MacPherson.
\newblock Equivariant cohomology, {K}oszul duality, and the localization
  theorem.
\newblock {\em Invent. Math.}, 131(1):25--83, 1998.

\bibitem[Kly85]{klyachko}
A.~A. Klyachko.
\newblock Orbits of a maximal torus on a flag space.
\newblock {\em Funktsional. Anal. i Prilozhen.}, 19(1):77--78, 1985.

\bibitem[KM05]{km}
Allen Knutson and Ezra Miller.
\newblock Gr\"obner geometry of {S}chubert polynomials.
\newblock {\em Ann. of Math. (2)}, 161(3):1245--1318, 2005.

\bibitem[KMS06]{kms}
Allen Knutson, Ezra Miller, and Mark Shimozono.
\newblock Four positive formulae for type {$A$} quiver polynomials.
\newblock {\em Invent. Math.}, 166(2):229--325, 2006.

\bibitem[KMY09]{kmy}
Allen Knutson, Ezra Miller, and Alexander Yong.
\newblock Gr\"obner geometry of vertex decompositions and of flagged tableaux.
\newblock {\em J. Reine Angew. Math.}, 630:1--31, 2009.

\bibitem[Knu10]{knutson}
Allen Knutson.
\newblock Puzzles, positroid varieties, and equivariant k-theory of
  grassmannians, 2010.

\bibitem[Mac95]{macdonald}
Ian~G. Macdonald.
\newblock {\em Symmetric functions and {H}all polynomials}.
\newblock Oxford Mathematical Monographs. The Clarendon Press, Oxford
  University Press, New York, second edition, 1995.
\newblock With contributions by A. Zelevinsky, Oxford Science Publications.

\bibitem[MS05]{miller-sturmfels}
Ezra Miller and Bernd Sturmfels.
\newblock {\em Combinatorial commutative algebra}, volume 227 of {\em Graduate
  Texts in Mathematics}.
\newblock Springer-Verlag, New York, 2005.

\bibitem[Spe09]{speyer}
David~E. Speyer.
\newblock A matroid invariant via the {$K$}-theory of the {G}rassmannian.
\newblock {\em Adv. Math.}, 221(3):882--913, 2009.

\bibitem[Whi77]{white-normality}
Neil~L. White.
\newblock The basis monomial ring of a matroid.
\newblock {\em Advances in Math.}, 24(3):292--297, 1977.

\bibitem[WY14]{wy}
Benjamin~J. Wyser and Alexander Yong.
\newblock Polynomials for {${\rm GL}_p\times{\rm GL}_q$} orbit closures in the
  flag variety.
\newblock {\em Selecta Math. (N.S.)}, 20(4):1083--1110, 2014.

\end{thebibliography}

\end{document}